\documentclass[a4paper,portuguese,twoside,thmsa,12pt]{article}
\usepackage{amsmath}
\usepackage{amsfonts}
\usepackage[latin1]{inputenc}
\usepackage{amssymb}
\usepackage{latexsym}
\usepackage{makeidx}
\usepackage{indentfirst}
\usepackage{amscd}
\usepackage{eufrak}
\usepackage[latin1]{inputenc}
\usepackage{amsmath}
\usepackage[dvips]{graphicx}
\usepackage[dvips]{epsfig}
\usepackage{graphicx,color}

\DeclareMathOperator{\supp}{supp}

\newtheorem{theorem}{Theorem}

\newtheorem{corollary}[theorem]{Corollary}

\newtheorem{definition}[theorem]{Definition}

\newtheorem{lemma}[theorem]{Lemma}

\newtheorem{proposition}[theorem]{Proposition}
\newtheorem{remark}[theorem]{Remark}

\newenvironment{proof}[1][Proof]{\noindent\textbf{#1:} }{\hfill \rule{0.5em}{0.5em}}

\begin{document}

\title{A dynamical condition for differentiability of Mather's average action%
}
\author{Alexandre Rocha and Mário J. D. Carneiro}
\maketitle

\begin{abstract}
We prove the differentiability of 
$\beta $ of Mather function on all homology classes corresponding to
rotation vectors of measures whose supports are contained in a Lipschitz
Lagrangian absorbing graph, invariant by Tonelli Hamiltonians. We also show
the relationship between local differentiability of $\beta $ and local
integrability of the Hamiltonian flow.
\end{abstract}

\section{Introduction}

Given a Tonneli Lagrangian $L$, Mather introduced the $\beta $-function of $%
L $, which is a convex and superlinear function. Many interesting properties
of the Euler--Lagrange flow can be derived from the study of the behaviour
of the $\beta $-function. Understanding whether or not this function is
differentiable and what are the implications of its re-gularity to the
dynamics of the system is an interesting problem. This type of problem was
developed by D. Massart in several works as, for example, \cite{mas1} and
\cite{mas2}.

Even in this context, D. Massart and A. Sorrentino get in the work \textit{%
Differentiability of Mather's average action and integrability on closed
surfaces }(see \cite{sor1}) the relation, on closed surfaces, between the
differentiability of $\beta $-function and the integrability of the system.
However there are examples of systems which are not integrable but have
invariant Lipschitz Lagrangian graphs, i.e. invariant graphs of the form $%
\mathcal{G}_{\eta ,u}=graph\left( \eta +du\right) $ where $\eta $ is a
closed one-form and $u$ is a function of class $C^{1}$ with Lipschitz
differential.

Motivated by these problems, in this work we study the differentiability of $%
\beta $ at homologies $h$ whose the measures with vector rotation $h$ are
supported on an invariant Lipschitz Lagrangian graph. We obtain
differentiability of $\beta $ in these homologies if the invariant graph is
an absorbing graph, i.e. a graph which not contain $\omega $-limit of
minimizing curves out of it\footnote[1]{%
the formal definitions and all notations are defined in the Sections \ref%
{sec2}, \ref{sec3} and \ref{sec4}}. More precisely, we prove the following
theorem:

\begin{theorem}
\label{teo1}Let $\mathcal{G}_{\eta ,u}$ be an invariant Lipschitz Lagrangian
graph. Then $\mathcal{G}_{\eta ,u}$ is absor-bing if and only if $\beta $ is
differentiable at $h$ for all $h\in \partial \alpha \left( \left[ \eta %
\right] \right) $ and $\mathcal{A}_{\left[ \eta \right] }^{\ast }=\mathcal{G}%
_{\eta ,u}.$
\end{theorem}

One can derive some consequences of this result. For instance, if the system
is locally Lipschitz integrable on an invariant Lipschitz Lagrangian graph $%
\mathcal{G}_{\eta ,u}\subset T^{\ast }M,$ i.e. there exists a neighborhood $%
\mathcal{V}$ of $\mathcal{G}_{\eta ,u}$ in $T^{\ast }M$ foliated by disjoint
invariant Lipschitz Lagrangian graphs, of course that the graphs contained
in $\mathcal{V}$ are absorbing, so the following result is a local version
of a result of D. Massart and A. Sorrentino (See \cite{sor1}, Lemma 5).

\begin{theorem}
\label{teo2}Let $\mathcal{G}_{\eta ,u}$ be an invariant Lipschitz Lagrangian
graph. If $H$ is locally Lipschitz integrable on $\mathcal{G}_{\eta ,u}$,
then there exists a neighborhood $U_{0}\subset H^{1}\left( M;%
\mathbb{R}
\right) $ of $\left[ \eta \right] $ such that $\beta $ is differentiable at
any point of $V=\bigcup\nolimits_{c\in U_{0}}\partial \alpha \left( c\right)
.$
\end{theorem}

We prove the converse of Theorem \ref{teo2} in the case $M$ equals torus $%
\mathbb{T}^{2}$ (see Theorem \ref{teo4}). In this case, the set $%
V=\bigcup\nolimits_{c\in U_{0}}\partial \alpha \left( c\right) ,$ obtained
in the above statement, is open in $H_{1}\left( \mathbb{T}^{2};%
\mathbb{R}
\right) $. Then we geralize (\cite{sor1}, Theorem 3) to local case.

We also give a particular attention to existence of neighborhood contained
in the tiered Mañé, introduced by M-C. Arnaud in \cite{arn1}, and its
relation with the local integrability of system and therefore with the local
differentiability of $\beta .$ Indeed, we prove a local version of a result
of M-C. Arnaud (See \cite{arn2}, Theorem 1), in the Section \ref{sec1},
Corollary \ref{coro2}.

\section{Preliminaries\label{sec2}}

Let $M$ be a compact connected manifold and $TM$ its tangent bundle. A
Tonelli's Lagrangian is a function $L:TM\rightarrow
\mathbb{R}
$ of class at least $C^{2}$ which is convex and superlinear. Let us recall
the main concepts introduced by Mather in \cite{mat1}. Let $\mathcal{M}%
\left( L\right) $ be the set of probabilities on the Borel $\sigma $-algebra
on $TM$ which are invariant under the Euler--Lagrange flow $\varphi _{t}^{L}$%
. The Euler Lagrange flow generated by $L$ does not change by adding a
closed one form $\eta $ and the action of a probability measure $\mu \in
\mathcal{M}\left( L\right) $, defined by%
\begin{equation*}
A_{L-c}\left( \mu \right) =\int_{TM}\left( L-\eta \right) d\mu
\end{equation*}%
depends only on the cohomology class $c=[\eta ]\in H^{1}(M;{\mathbb{R}})$.

The minimal action value, which also depends only on the cohomology class $%
c=[\eta ],$ is denoted by $-\alpha (c),$ that is:%
\begin{equation*}
\alpha (c)=-\inf_{\mu \in \mathcal{M}\left( L\right) }A_{L-c}\left( \mu
\right) .
\end{equation*}%
Mather proved that the function $c\mapsto \alpha \left( c\right) ,$
so-called $\alpha $\textit{\ of Mather function}, is convex and superlinear.
It is known that $\alpha (c)$ is the energy level that contains the \textit{%
Mather set for the cohomology class }$c$:%
\begin{equation*}
{\widetilde{\mathcal{M}}}_{c}=\overline{\bigcup_{\mu }\supp(\mu )},
\end{equation*}%
where the union is taken over the set of Borel probability measures $\mu \in
\mathcal{M}\left( L\right) $ called $c$\textit{-minimizing,} i.e. $\alpha
(c)=-A_{L-c}\left( \mu \right) .$ The set ${\widetilde{\mathcal{M}}}_{c}$ is
a compact invariant set which is a graph over a compact subset $\mathcal{M}%
_{c}$ of $M$, the projected Mather set (see \cite{mat1}). $\mathcal{M}_{c}$
is laminated by curves, which are global (or time independent) minimizers.

Given a probability measure $\mu \in \mathcal{M}\left( L\right) $, its
\textit{homology} or its \textit{rotation vector} is defined as the unique $%
\rho \left( \mu \right) \in H_{1}(M;{\mathbb{R}})$ such that
\begin{equation*}
\left\langle \rho (%
\mu
),\left[ \omega \right] \right\rangle =\int_{TM}\omega d\mu ,
\end{equation*}%
for all closed 1-forms $\omega $ on $M.$ By convexity, we can consider the
dual Fenchel of $\alpha $, called $\beta $ \textit{of Mather function}, as%
\begin{equation*}
\beta \left( h\right) =\inf_{\rho (%
\mu
)=h}A_{L}\left( \mu \right) .
\end{equation*}%
Mather also proved that the $\beta $ function is convex and superlinear. We
say that a measure $\mu \in \mathcal{M}\left( L\right) $ with $\rho \left(
\mu \right) =h$ is $h$\textit{-minimizing }if $\beta \left( h\right)
=A_{L}\left( \mu \right) .$ The set $\widetilde{\mathcal{M}}^{h}$ is the
union of supports of probability measures $h$-minimizing.

In general, the maps $\alpha $ and $\beta $ are neither strictly convex, nor
differentiable. The projection on domain of regions of graph where either
map is affine are called flats. Actually, if the map is strictly convex at a
point, the flat is this only point and if the map is not strictly convex,
the flat is non-trivial. By duality we have the inequality%
\begin{equation*}
\alpha (c)+\beta \left( h\right) \geq \left\langle c,h\right\rangle ,\forall
c\in H^{1}(M;{\mathbb{R}}),\forall h\in H_{1}(M;{\mathbb{R}}),
\end{equation*}%
called Fenchel inequality. Given $c\in H^{1}(M;{\mathbb{R}})$ (resp. $h\in
H_{1}(M;{\mathbb{R}})$), the homology class $h\in H_{1}(M;{\mathbb{R}})$
(resp. $c\in H^{1}(M;{\mathbb{R}})$) achieving equality in the Fenchel
inequality is called subderivative of $\alpha $ in $c$ (resp. subderivative
of $\beta $ in $h$)$.$ The set composed by subderivatives of $\alpha $ in $c$
(resp. subderivatives of $\beta $ in $h$) is called Legendre transform of $c$
(resp. $h$), and denoted $\partial \alpha \left( c\right) $ (resp. $\partial
\beta \left( h\right) $). Therefore, $\partial \alpha \left( c\right) $ is a
flat of $\beta $ and $\partial \beta \left( h\right) $ is a flat of $\alpha
. $ By convexity, the sets $\partial \alpha \left( c\right) $ and $\partial
\beta \left( h\right) $ are non-empty.

Many interesting properties of the Euler--Lagrange flow can be derived from
the study of the behaviour of the $\beta $-function. For instance, if $h$ is
an\textit{\ extremal point} of $\beta $-function, i.e. $h$ is not convex
combination of two elements in a same flat of $\beta ,$ then there exist
ergodic measures with homology $h$ (see \cite{man1}).

Let us recall that we can associate to such a Tonelli's Lagrangian $L$ the
Hamiltonian function $H:T^{\ast }M\rightarrow
\mathbb{R}
$ via Legendre transform $\mathcal{L}:TM\rightarrow T^{\ast }M,$ which under
our assumption, is a diffeomorphism of class at least $C^{1}$, defined in
coordinates by%
\begin{equation*}
\mathcal{L}\left( x,v\right) =\left( x,\frac{\partial L}{\partial v}\left(
x,v\right) \right) .
\end{equation*}%
Actually, $H$ is the dual Fenchel of $L$ and also is convex and superlinear.
Given a cohomology class $c$ and a closed 1-form $\eta _{c}$ with $\left[
\eta _{c}\right] =c,$ we consider the Hamilton-Jacobi equation%
\begin{equation}
H\left( x,\eta _{c}\left( x\right) +d_{x}u\right) =\alpha (c).  \tag{HJ}
\end{equation}%
A Lipschitz function $u:M\rightarrow
\mathbb{R}
$ is called a subsolution of Hamilton-Jacobi for the Lagrangian $L-c$ if for
some closed 1-form $\eta _{c}$ with $\left[ \eta _{c}\right] =c,$ we have
\begin{equation}
H\left( x,\eta _{c}\left( x\right) +d_{x}u\right) \leq \alpha (c),
\label{fo5}
\end{equation}%
at almost every point. Note that this definition is equivalent to the notion
of viscosity subsolutions (see \cite{fa1}). We denote by $C^{1,1}$ the set
of differentiable functions with Lipschitz differential. Observe that a $%
C^{1,1}$ function $u$ is solution of (HJ) if and only if the graph of $\eta
_{c}+du,$ denoted by $\mathcal{G}_{\eta _{c},u},$ is invariant under
Hamiltonian flow.

We now recall the definition of calibrated curves (see \cite{fa1}). If $%
u:M\rightarrow
\mathbb{R}
$ is a subsolution of Hamilton-Jacobi for $L-c$, we say that the curve $%
\gamma :I\rightarrow M$ is $\left( u,L-c,\alpha \left( c\right) \right) $%
-calibrated if, for the representative $\eta _{c}$ of the cohomology class $%
c $ given in $\left( \ref{fo5}\right) $, we have the equality%
\begin{equation*}
u\left( \gamma \left( t\right) \right) -u\left( \gamma \left( t^{\prime
}\right) \right) =\int\nolimits_{t^{\prime }}^{t}L\left( \gamma \left(
s\right) ,\dot{\gamma}\left( s\right) \right) -\eta _{c}\left( \dot{\gamma}%
\left( s\right) \right) +\alpha \left( c\right) ds,
\end{equation*}%
for all $t^{\prime },t\in I.$ The subset $\widetilde{\mathcal{I}}_{c}(u)$ of
$TM$ is defined by%
\begin{equation*}
\widetilde{\mathcal{I}}_{c}(u)=\left\{ \left( x,v\right) \in TM:\gamma
_{\left( x,v\right) }\text{ is }\left( u,L-c,\alpha \left( c\right) \right)
\text{-calibrated}\right\} ,
\end{equation*}%
where $\gamma _{\left( x,v\right) }=\pi \circ \varphi _{t}^{L}\left(
x,v\right) $. The set $\widetilde{\mathcal{I}}_{c}(u)$ is invariant and the
curves contained in it are called \textit{curves }$c$\textit{-minimizing}.

Using the sets $\widetilde{\mathcal{I}}_{c}(u)$, one can give (see \cite{fa2}%
) the following characterization of the Mañé set and of the Aubry set:%
\begin{equation*}
\widetilde{\mathcal{N}}_{c}=\bigcup\limits_{u\in SS_{c}}\widetilde{\mathcal{I%
}}_{c}\left( u\right) \text{ and }\widetilde{\mathcal{A}}_{c}=\bigcap%
\limits_{u\in SS_{c}}\widetilde{\mathcal{I}}_{c}\left( u\right) ,
\end{equation*}%
where $SS_{c}$ is the set of subsolution of (HJ) for $L-c$. These invariant
sets contain the Mather set and have interesting dynamical properties, for
instance $\widetilde{\mathcal{A}}_{c}$ also is graph whose projection is
laminated by global minimizers and it is chain recurrent. The Mañé set $%
\widetilde{\mathcal{N}}_{c}$ is connected and chain transitive (see for
instance \cite{gon2}).

Using the duality between Lagrangian and Hamiltonian, via Legendre
transform, we define the sets of Mather, Aubry and Mañé in the cotangent
bundle, respectively by%
\begin{equation*}
\mathcal{M}_{c}^{\ast }=\mathcal{L}\left( \widetilde{\mathcal{M}}_{c}\right)
,\mathcal{A}_{c}^{\ast }=\mathcal{L}\left( \widetilde{\mathcal{A}}%
_{c}\right) \text{ e }\mathcal{N}_{c}^{\ast }=\mathcal{L}\left( \widetilde{%
\mathcal{N}}_{c}\right) .
\end{equation*}%
One useful way to produce invariant Lipschitz Lagrangian graphs is to show
that $\pi \left( \mathcal{A}_{c}^{\ast }\right) =M.$ If this is the case,
the Theorem 2.5 of \cite{fa2} says that there exists an unique solution $u$
of (HJ) for the Lagrangian $L-c$ which is $C^{1,1}$ and such that $\mathcal{A%
}_{c}^{\ast }$ is the graph of $\eta _{c}+du,$ for some $\eta _{c}$
representative of cohomology class $c.$

\section{Absorbing sets\label{sec3}}

If $u:M\rightarrow
\mathbb{R}
$ is a subsolution of (HJ) for $L-c$, we denote by $\widetilde{\mathcal{I}}%
_{c}^{+}\left( u\right) $ the subset of $TM$ defined as
\begin{equation*}
\widetilde{\mathcal{I}}_{c}^{+}\left( u\right) =\left\{ \left( x,v\right)
:\gamma _{\left( x,v\right) }|_{\left[ 0,+\infty \right) }\text{ is }\left(
u,L-c,\alpha \left( c\right) \right) \text{-calibrated}\right\} ,
\end{equation*}%
where $\gamma _{\left( x,v\right) }$ is the curve defined on $%
\mathbb{R}
$ by%
\begin{equation*}
\gamma _{\left( x,v\right) }\left( t\right) =\pi \circ \varphi
_{t}^{L}\left( x,v\right) .
\end{equation*}%
The \textit{forward Mañe set} is defined by%
\begin{equation*}
\widetilde{\mathcal{N}}_{c}^{+}=\bigcup\limits_{u\in SS_{c}}\widetilde{%
\mathcal{I}}_{c}^{+}\left( u\right) ,
\end{equation*}%
where $SS_{c}$ is the set of critical subsolutions for the Lagrangian $L-c.$
We define the \textit{forward tiered Mañé set} as the union of all forward Ma%
ñe sets, i.e. the subset of $TM$ given by
\begin{equation*}
\mathcal{N}_{+}^{T}\left( L\right) =\bigcup\limits_{c\in H^{1}\left(
M\right) }\widetilde{\mathcal{N}}_{c}^{+}.
\end{equation*}

\begin{definition}
We say that an invariant set $\Lambda \subset TM$ is an absorbing set if for
all $\left( x,v\right) \in \mathcal{N}_{+}^{T}\left( L\right) $ we have
\begin{equation*}
\omega \left( x,v\right) \subset \Lambda \Rightarrow \left( x,v\right) \in
\Lambda
\end{equation*}
\end{definition}

\begin{definition}
\label{def1}Let $\mathcal{G\subset }T^{\ast }M$ be an invariant Lipschitz
Lagrangian graph. We say that $\mathcal{G}$ is an absorbing graph if $%
\mathcal{L}^{-1}\left( \mathcal{G}\right) $ is an absorbing set.
\end{definition}

\begin{lemma}
\label{lem3}If $\left( x,v\right) \in \widetilde{\mathcal{N}}_{c}^{+},$ then
the $\omega $-limit set $\omega \left( x,v\right) $ is contained in $%
\widetilde{\mathcal{A}}_{c}.$
\end{lemma}

\begin{proof}
In fact, let $\left( x,v\right) \in \widetilde{\mathcal{N}}_{c}^{+}$ and $%
\gamma :%
\mathbb{R}
\rightarrow M\ $the projection of Euler-Lagrange solution $\gamma \left(
t\right) =\pi \circ \varphi _{t}^{L}\left( x,v\right) $ curve such that $%
\gamma |_{\left[ 0,+\infty \right) }$ is $\left( \overline{u},L-c,\alpha
\left( c\right) \right) $-calibrated for some $\overline{u}\in $ $SS_{c}.$
This means that there exists a closed 1-form $\eta _{c}$ with $\left[ \eta
_{c}\right] =c$ such that $H\left( x,\eta _{c}\left( x\right) +d_{x}%
\overline{u}\right) \leq \alpha \left( c\right) $ at almost every point and
\begin{equation*}
\overline{u}\left( \gamma \left( t\right) \right) -\overline{u}\left( \gamma
\left( t^{\prime }\right) \right) =A_{L-\eta _{c}+\alpha \left( c\right)
}\left( \gamma |_{\left[ t^{\prime },t\right] }\right) \text{ for all }%
0<t^{\prime }<t.
\end{equation*}%
Let $\left( y,z\right) \in \omega \left( x,v\right) ,$ i.e. $\left(
y,z\right) =\lim_{n\rightarrow \infty }\left( \gamma ,\dot{\gamma}\right)
\left( t_{n}\right) $ with $t_{n}\rightarrow \infty .$

Let us consider $\sigma :%
\mathbb{R}
\rightarrow M\ $the projection of the Euler-Lagrange solution $\sigma \left(
t\right) =\pi \circ \varphi _{t}^{L}\left( y,z\right) .$ Given a
Hamilton-Jacobi subsolution $u:M\rightarrow
\mathbb{R}
$ belonging to $SS_{c}$, there exists a closed 1-form $\xi _{c}$ with $\left[
\xi _{c}\right] =c$ and $H\left( x,\xi _{c}\left( x\right) +d_{x}u\right)
\leq \alpha \left( c\right) $ at almost every point. It follows from (\cite%
{fa1}, Proposition 4.2.3) that $u$ satisfies
\begin{equation*}
u\left( \gamma \left( t\right) \right) -u\left( \gamma \left( t^{\prime
}\right) \right) \leq A_{L-\xi _{c}+\alpha \left( c\right) }\left( \gamma
|_{ \left[ t^{\prime },t\right] }\right) ,\forall t^{\prime }<t.
\end{equation*}%
We can consider $V$ a $C^{\infty }\left( M\right) $ function such that $\eta
_{c}=\xi _{c}+dV.$ Thus, if $s>0$ and $\left( t_{k}\right) $ and $\left(
t_{m}\right) $ are two subsequences of $\left( t_{n}\right) $ such that $%
t_{k}-s>t_{m}+s$ and $t_{k},t_{m}>s,$ we have%
\begin{align*}
& A_{L-\xi _{c}+\alpha \left( c\right) }\left( \sigma |_{\left[ -s,s\right]
}\right) +u\left( \sigma \left( -s\right) \right) -u\left( \sigma \left(
s\right) \right) \\
& =\lim_{m,k}\left[ A_{L-\xi _{c}+\alpha \left( c\right) }\left( \gamma |_{%
\left[ t_{m}-s,t_{m}+s\right] }\right) +u\left( \gamma \left( t_{k}-s\right)
\right) -u\left( \gamma \left( t_{m}+s\right) \right) \right] \\
& \leq \lim_{m,k}\left[ A_{L-\xi _{c}+\alpha \left( c\right) }\left( \gamma
|_{\left[ t_{m}-s,t_{m}+s\right] }\right) +A_{L-\xi _{c}+\alpha \left(
c\right) }\left( \gamma |_{\left[ t_{m}+s,t_{k}-s\right] }\right) \right] \\
& =\lim_{m,k}A_{L-\xi _{c}+\alpha \left( c\right) }\left( \gamma |_{\left[
t_{m}-s,t_{k}-s\right] }\right) \\
& =\lim_{m,k}\left[ A_{L-\xi _{c}-dV+\alpha \left( c\right) }\left( \gamma
|_{\left[ t_{m}-s,t_{k}-s\right] }\right) +V\left( \gamma \left(
t_{k}-s\right) \right) -V\left( \gamma \left( t_{m}-s\right) \right) \right]
\\
& =\lim_{m,k}A_{L-\eta _{c}+\alpha \left( c\right) }\left( \gamma |_{\left[
t_{m}-s,t_{k}-s\right] }\right) +V\left( \gamma \left( -s\right) \right)
-V\left( \gamma \left( -s\right) \right) \\
& =\lim_{m,k}\overline{u}\left( \gamma \left( t_{k}-s\right) \right) -%
\overline{u}\left( \gamma \left( t_{m}-s\right) \right) \\
& =\overline{u}\left( \gamma \left( -s\right) \right) -\overline{u}\left(
\gamma \left( -s\right) \right) =0.
\end{align*}%
Therefore%
\begin{equation*}
A_{L-\xi _{c}+\alpha \left( c\right) }\left( \sigma |_{\left[ -s,s\right]
}\right) \leq u\left( \sigma \left( s\right) \right) -u\left( \sigma \left(
-s\right) \right) .
\end{equation*}%
The opposite inequality holds because $u$ is a Hamilton-Jacobi subsolution
(see \cite{fa1}, Proposition 4.2.3). This shows that
\begin{equation*}
\left( y,z\right) \in \bigcap\limits_{u\in SS_{c}}\widetilde{\mathcal{I}}%
_{c}\left( u\right) =\widetilde{\mathcal{A}}_{c}.
\end{equation*}
\end{proof}

Examples of absorbing graphs are the so-called \textit{Schwartzman strictly
ergodic graphs }(see \cite{fa3}), i.e. invariant graphs $\Lambda $ which
support an invariant measure with full support. In fact, let $\mu $ the
invariant measure supported in $\Lambda $ with $\rho \left( \mu \right) =h.$
If $\omega \left( x,p\right) \subset \Lambda $ for some $\left( x,p\right)
\in \mathcal{L}\left( \widetilde{\mathcal{I}}_{c}^{+}\left( L\right) \right)
,$ then it follows from above lemma that $\omega \left( x,p\right) \subset
\mathcal{A}_{c}^{\ast }\cap \Lambda .$ In particular,
\begin{equation*}
\omega \left( x,p\right) \subset \mathcal{L}\left( \widetilde{\mathcal{M}}%
^{h}\right) \cap \mathcal{A}_{c}^{\ast }.
\end{equation*}%
Since $\pi \left( \widetilde{\mathcal{M}}^{h}\right) =M,$ we have $\mathcal{A%
}_{c}^{\ast }=\mathcal{L}\left( \widetilde{\mathcal{M}}^{h}\right) =\Lambda
. $ By the graph property, we conclude that $\left( x,p\right) \in \Lambda .$

Actually, the same argument can be used to prove that an invariant graph $%
\Lambda $ such that all invariant probability measures with support
contained in $\Lambda $ have the same rotation vector $h$ and the union of
their supports equals $\Lambda ,$ also it is absorbing.

\begin{proposition}
\label{pro3}Let $\Lambda \subset TM$ be an invariant absorbing set. If $%
\widetilde{\mathcal{A}}_{c}\subset \Lambda $ for some $c\in H^{1}\left( M;%
\mathbb{R}
\right) ,$ then $\Lambda $ projects onto the whole manifold $M.$
\end{proposition}

\begin{proof}
Let us consider a closed 1-form $\eta _{c}$ representative of the cohomology
class $c.$ It follows from (\cite{fa1}, Theorem 4.9.3) that there exists a
weak KAM of positive type $u_{+}$ for the Lagrangian $L-\eta _{c}$. Then $%
u_{+}$ is subsolution of (HJ) and given $x\in M$ we can find a $C^{1}$ curve
$\gamma _{x}:\left[ 0,\infty \right) \rightarrow M$ with $\gamma _{x}\left(
0\right) =x,$ which is $\left( u_{+},L-\eta _{c},\alpha \left( c\right)
\right) $-calibrated. This means that for all $0<t^{\prime }<t$ holds%
\begin{equation*}
u_{+}\left( \gamma _{x}\left( t\right) \right) -u_{+}\left( \gamma
_{x}\left( t^{\prime }\right) \right) =\int_{t^{\prime }}^{t}L\left( \gamma
_{x}\left( s\right) ,\dot{\gamma}_{x}\left( s\right) \right) -\eta
_{c}\left( \dot{\gamma}_{x}\left( s\right) \right) +\alpha \left( c\right)
ds.
\end{equation*}%
Therefore $\left( \gamma _{x},\dot{\gamma}_{x}\right) \in \widetilde{%
\mathcal{I}}_{c}^{+}\left( u_{+}\right) \subset \widetilde{\mathcal{N}}%
_{c}^{+}$ and, by Lemma \ref{lem3}, we have that the $\omega $-limit set of $%
\left( \gamma _{x},\dot{\gamma}_{x}\right) $ is contained in Aubry set $%
\widetilde{\mathcal{A}}_{c}.$ Since $\Lambda $ is an absorbing set that
contains $\widetilde{\mathcal{A}}_{c},$ we obtain $\left( x,\dot{\gamma}%
_{x}\left( 0\right) \right) \in \Lambda .$
\end{proof}

\begin{lemma}
\label{lem1}Let $c\in H^{1}\left( M;%
\mathbb{R}
\right) $ and $h\in H_{1}\left( M;%
\mathbb{R}
\right) .$ We have $\widetilde{\mathcal{M}}^{h}\subset \widetilde{\mathcal{M}%
}_{c}$ if and only if $c\in \partial \beta \left( h\right) .$
\end{lemma}

\begin{proof}
If $\widetilde{\mathcal{M}}^{h}\subset \widetilde{\mathcal{M}}_{c},$ then
there exists a $c$-minimizing measure $\mu $ with $\rho \left( \mu \right)
=h.$ So%
\begin{equation*}
-\alpha \left( c\right) =\int\nolimits_{TM}\left( L-\eta _{c}\right) d\mu
=\int\nolimits_{TM}Ld\mu -\left\langle c,h\right\rangle =\beta \left(
h\right) -\left\langle c,h\right\rangle ,
\end{equation*}%
where $\eta _{c}$ is a representative of the cohomology class $c.$ This show
that $c\in \partial \beta \left( h\right) .$

Conversely, let $\mu $ minimizing measure with rotation vector $h$. If $c\in
\partial \beta \left( h\right) ,$ then $\beta \left( h\right) =\left\langle
c,h\right\rangle -\alpha \left( c\right) .$ Therefore
\begin{equation*}
-\alpha \left( c\right) =\beta \left( h\right) -\left\langle
c,h\right\rangle =\int\nolimits_{TM}Ld\mu -\left\langle c,\rho \left( \mu
\right) \right\rangle =\int\nolimits_{TM}\left( L-\eta _{c}\right) d\mu .
\end{equation*}%
This proves that $%
\mu
$ is $c$-minimizing.
\end{proof}

\section{Proof of Theorems \protect\ref{teo1} and \protect\ref{teo2} \label%
{sec4}}

In order to prove Theorems \ref{teo1} and \ref{teo2} we begin by proving the
following lemma:

\begin{lemma}
\label{lem4}Let $\mathcal{G}_{\eta ,u}$ be an invariant Lipschitz Lagrangian
absorbing graph. Then $\mathcal{N}_{\left[ \omega \right] }^{\ast }\cap
\mathcal{G}_{\eta ,u}\neq \emptyset $ if and only if $\left[ \omega \right] =%
\left[ \eta \right] .$ Moreover, $\mathcal{A}_{\left[ \eta \right] }^{\ast }=%
\mathcal{G}_{\eta ,u}.$
\end{lemma}

\begin{proof}
Since $\mathcal{G}_{\eta ,u}$ is an invariant Lipschitz Lagrangian graph, it
is contained in the Mañé's set $\mathcal{N}_{\left[ \eta \right] }^{\ast }.$
Therefore, if $\left( x_{0},p\right) \in \mathcal{N}_{\left[ \omega \right]
}^{\ast }\cap \mathcal{G}_{\eta ,u}$, then we have $\left( x_{0},p\right)
\in \mathcal{N}_{\left[ \omega \right] }^{\ast }\cap \mathcal{N}_{\left[
\eta \right] }^{\ast }.$ By Lemma \ref{lem3}, the $\omega $-limit set of
points in Mañé set is contained in the Aubry set. Thus%
\begin{equation*}
\omega \left( x_{0},p\right) \subset \mathcal{A}_{\left[ \omega \right]
}^{\ast }\cap \mathcal{A}_{\left[ \eta \right] }^{\ast }.
\end{equation*}

This shows that the intersection $\mathcal{A}_{\left[ \omega \right] }^{\ast
}\cap \mathcal{A}_{\left[ \eta \right] }^{\ast }$ is non-empty. Then, by a
result of D. Massart (see \cite{mas1}, Proposition 6), $\alpha $ has a flat $%
F$ containing $\left[ \omega \right] $ and $\left[ \eta \right] $. Let $c$
be a cohomology class belonging to the relative interior of $F.$ It follows
from (\cite{mas1}, Proposition 6) that $\mathcal{A}_{c}^{\ast }\subset
\mathcal{A}_{\left[ \omega \right] }^{\ast }\cap \mathcal{A}_{\left[ \eta %
\right] }^{\ast }.$

Let us consider a closed 1-form $\eta _{c},$ with $\left[ \eta _{c}\right]
=c.\ $It follows from (\cite{fa1}, Theorem 4.9.3) that there exists a weak
KAM of positive type $u_{+}$ for the Lagrangian $L-\eta _{c}$. Then $u_{+}$
is subsolution of (HJ) and given $x\in M,$ we can find a $C^{1}$ curve $%
\gamma _{x}:\left[ 0,\infty \right) \rightarrow M$ with $\gamma _{x}\left(
0\right) =x,$ which is $\left( u_{+},L-\eta _{c},\alpha \left( c\right)
\right) $-calibrated. This means that for all $0<t^{\prime }<t$ holds%
\begin{equation*}
u_{+}\left( \gamma _{x}\left( t\right) \right) -u_{+}\left( \gamma
_{x}\left( t^{\prime }\right) \right) =\int_{t^{\prime }}^{t}L\left( \gamma
_{x}\left( s\right) ,\dot{\gamma}_{x}\left( s\right) \right) -\eta
_{c}\left( \dot{\gamma}_{x}\left( s\right) \right) +\alpha \left( c\right)
ds.
\end{equation*}%
Therefore $\left( \gamma _{x},\dot{\gamma}_{x}\right) \in \widetilde{%
\mathcal{N}}_{c}^{+}$ and, by Lemma \ref{lem3}, we have that the $\omega $%
-limit set of $\left( \gamma _{x},\dot{\gamma}_{x}\right) $ is contained in
Aubry set $\widetilde{\mathcal{A}}_{c}.$ Recall that by invariance of $%
\mathcal{G}_{\eta ,u}$ we have $\mathcal{G}_{\eta ,u}=\mathcal{L}\left(
\widetilde{\mathcal{I}}_{\eta }\left( u\right) \right) $ and the Aubry set $%
\mathcal{A}_{\left[ \eta \right] }^{\ast }$ is contained in $\mathcal{G}%
_{\eta ,u}.$ Since $\mathcal{A}_{c}^{\ast }\subset \mathcal{A}_{\left[ \eta %
\right] }^{\ast }\subset \mathcal{G}_{\eta ,u},$ we conclude the $\omega $%
-limit set of $\left( \gamma _{x},\dot{\gamma}_{x}\right) $ is contained in $%
\mathcal{G}_{\eta ,u}.$ Moreover, it follows from $\mathcal{G}_{\eta ,u}$
being an absorbing graph, that the Hamiltonian orbit $\mathcal{L}\left(
\gamma _{x},\dot{\gamma}_{x}\right) $ is entirely contained in $\mathcal{G}%
_{\eta ,u}.$ As a consequence, given $T>0,$ we have
\begin{equation*}
\mathcal{L}\left( \gamma _{x}\left( -T\right) ,\dot{\gamma}_{x}\left(
-T\right) \right) \in \mathcal{G}_{\eta ,u}.
\end{equation*}%
If we let $z=\gamma _{x}\left( -T\right) \in M,$ since $u_{+}$ is weak KAM,
there exists a $C^{1}$ curve $\gamma _{z}:\left[ 0,\infty \right)
\rightarrow M$ which is $\left( u_{+},L-\eta _{c},\alpha \left( c\right)
\right) $-calibrated with $\gamma _{z}\left( 0\right) =z.$ Similarly as
above, the $\omega $-limit set of $\mathcal{L}\left( \gamma _{z},\dot{\gamma}%
_{z}\right) $ is contained in $\mathcal{G}_{\eta ,u}.$ Thus $\mathcal{L}%
\left( \gamma _{z}\left( 0\right) ,\dot{\gamma}_{z}\left( 0\right) \right)
\in \mathcal{G}_{\eta ,u}$ and, since $\gamma _{z}\left( 0\right) =\gamma
_{x}\left( -T\right) ,$ we have $\dot{\gamma}_{z}\left( 0\right) =\dot{\gamma%
}_{x}\left( -T\right) .$ It follows from the uniqueness of solutions that $%
\gamma _{z}\left( t\right) =\gamma _{x}\left( t-T\right) .$ This shows that $%
\gamma _{x}|_{\left[ -T,\infty \right) }$ is $\left( u_{+},L-c,\alpha \left(
c\right) \right) $-calibrated and, by (\cite{fa1}, Lemma 4.13.1), $u_{+}$ is
differentiable at $x$ with%
\begin{equation*}
d_{x}u_{+}=\frac{\partial L}{\partial v}\left( x,\dot{\gamma}_{x}\left(
0\right) \right) -\eta _{c}\left( \dot{\gamma}_{x}\left( 0\right) \right) .
\end{equation*}

Now let us consider another weak KAM $v_{+}$ for the Lagrangian $L-\eta
_{c}. $ There exists a $C^{1}$ curve $\delta _{x}:\left[ 0,\infty \right)
\rightarrow M$ with $\delta _{x}\left( 0\right) =x,$ which is $\left(
v_{+},L-\eta _{c},\alpha \left( c\right) \right) $-calibrated. Similarly, we
conclude that%
\begin{equation*}
d_{x}v_{+}=\frac{\partial L}{\partial v}\left( x,\dot{\delta}_{x}\left(
0\right) \right) -\eta _{c}\left( \dot{\gamma}_{x}\left( 0\right) \right) .
\end{equation*}

Moreover, the points $\mathcal{L}\left( \gamma _{x}\left( 0\right) ,\dot{%
\gamma}_{x}\left( 0\right) \right) $ and $\mathcal{L}\left( \delta
_{x}\left( 0\right) ,\dot{\delta}_{x}\left( 0\right) \right) $ belong to the
graph $\mathcal{G}_{\eta ,u}$ with $\gamma _{x}\left( 0\right) =\delta
_{x}\left( 0\right) =x.$ Thus $\dot{\gamma}_{x}\left( 0\right) =\dot{\delta}%
_{x}\left( 0\right) \ $and we conclude that $d_{x}u_{+}=d_{x}v_{+}.$ Since
this equality holds for all $x\in M$ and $M$ is connected, we have that $%
u_{+}$ differ of $v_{+}$ by a constant. By arbitrariness of the two weak
KAM, we conclude that any two weak KAM for the Lagrangian $L-\eta _{c}$
differ by a constant. It follows from (\cite{fa2}, Proposition 4.4) that $%
\mathcal{A}_{c}^{\ast }=\mathcal{N}_{c}^{\ast }.$ On the other hand, since $%
u_{+}$ is differentiable everywhere in $M$, by (\cite{fa1}, Lemma 4.13.1),
we have
\begin{equation}
H\left( x,\eta _{c}\left( x\right) +d_{x}u_{+}\right) =\alpha \left(
c\right) ,  \label{fo2}
\end{equation}%
This means that the graph $\mathcal{G}_{\eta _{c},u_{+}}$ is invariant, so $%
\mathcal{G}_{\eta _{c},u_{+}}\subset \mathcal{N}_{c}^{\ast }=\mathcal{A}%
_{c}^{\ast }.$ By the graph property, we have $\mathcal{A}_{c}^{\ast }=%
\mathcal{N}_{c}^{\ast }=\mathcal{G}_{\eta _{c},u_{+}}.$ As a consequence of $%
\mathcal{A}_{c}^{\ast }\subset \mathcal{A}_{\left[ \omega \right] }^{\ast
}\cap \mathcal{A}_{\left[ \eta \right] }^{\ast },\ $we obtain
\begin{equation*}
\mathcal{A}_{c}^{\ast }=\mathcal{A}_{\left[ \omega \right] }^{\ast }=%
\mathcal{A}_{\left[ \eta \right] }^{\ast }=\mathcal{G}_{\eta ,u}.
\end{equation*}%
We also have that the projected Aubry set $\mathcal{A}_{\left[ \omega \right]
}$ is the whole manifold $M,$ so there exists a function $v\in C^{1,1}$ such
that $\mathcal{A}_{\left[ \omega \right] }^{\ast }=\mathcal{G}_{\omega ,v}.$
Therefore
\begin{equation*}
\eta \left( x\right) +d_{x}u=\omega \left( x\right) +d_{x}v,\forall x\in
M\Rightarrow \left( \eta -\omega \right) \left( x\right) =d_{x}\left(
v-u\right) ,
\end{equation*}%
which implies $\left[ \omega \right] =\left[ \eta \right] .\medskip $
\end{proof}

We can now prove the main result stated in Introduction.

\begin{proof}
(of Theorem \ref{teo1}) Suppose that $\mathcal{G}_{\eta ,u}$ is an absorbing
graph. By the Lemma \ref{lem4}, $\mathcal{A}_{\left[ \eta \right] }^{\ast }=%
\mathcal{G}_{\eta ,u}.$ We need to show that $\beta $ is differentiable on $%
\partial \alpha \left( \left[ \eta \right] \right) .$ For a given $h\in
\partial \alpha \left( \left[ \eta \right] \right) ,$ it suffices to show
that $\partial \beta \left( h\right) =\left\{ \left[ \eta \right] \right\} .$
Let us assume that $\left[ \xi \right] \in \partial \beta \left( h\right) ,$
so we have $h\in \partial \alpha \left( \left[ \eta \right] \right) \cap
\partial \alpha \left( \left[ \xi \right] \right) .$ It follows from Lemma %
\ref{lem1} that
\begin{equation*}
\widetilde{\mathcal{M}}^{h}\subset \widetilde{\mathcal{A}}_{\left[ \eta %
\right] }\cap \widetilde{\mathcal{A}}_{\left[ \xi \right] }.
\end{equation*}%
This implies that $\mathcal{A}_{\left[ \eta \right] }^{\ast }\cap \mathcal{A}%
_{\left[ \xi \right] }^{\ast }\neq \emptyset .$ Moreover, since $\mathcal{A}%
_{\left[ \eta \right] }^{\ast }\subset \mathcal{G}_{\eta ,u}$, by Lemma \ref%
{lem4} we have $\left[ \xi \right] =\left[ \eta \right] $. Hence $\beta $ is
differentiable at $h.$

Conversely, suppose that $\beta $ is differentiable at all homology class $%
h\in \partial \alpha \left( \left[ \eta \right] \right) $ and $\mathcal{A}_{%
\left[ \eta \right] }^{\ast }=\mathcal{G}_{\eta ,u}.$ Let $\left( x,p\right)
\in \mathcal{N}_{+}^{T}\left( L\right) \ $such that $\omega \left(
x,p\right) \subset \mathcal{G}_{\eta ,u}$ and let us consider $c\in
H^{1}\left( M;%
\mathbb{R}
\right) $ and $v\in SS_{c}$ such that $\left( x,p\right) \in \mathcal{L}%
\left( \widetilde{\mathcal{I}}_{c}^{+}\left( v\right) \right) .$ It follows
from Lemma \ref{lem3} that $\omega \left( x,p\right) \subset \mathcal{A}%
_{c}^{\ast }.$ Thus%
\begin{equation*}
\omega \left( x,p\right) \subset \mathcal{A}_{c}^{\ast }\cap \mathcal{G}%
_{\eta ,u}=\mathcal{A}_{c}^{\ast }\cap \mathcal{A}_{\left[ \eta \right]
}^{\ast }.
\end{equation*}%
This implies that $\mathcal{A}_{c}^{\ast }\cap \mathcal{A}_{\left[ \eta %
\right] }^{\ast }\neq \emptyset $ and, by (\cite{mas1}, Proposition 6),
there exists a flat $F$ of $\alpha $ such that $c$ and $\left[ \eta \right] $
belong to $F.$ As a consequence, there exists $h\in H_{1}\left( M;%
\mathbb{R}
\right) $ such that $\left[ \eta \right] ,c\in \partial \beta \left(
h\right) $. Since $\beta $ is differentiable at all $h\in \partial \alpha
\left( \left[ \eta \right] \right) ,$ we obtain $\left[ \eta \right] =c.$ In
particular, $\left( x,p\right) \in \mathcal{L}\left( \widetilde{\mathcal{I}}%
_{\left[ \eta \right] }^{+}\left( v\right) \right) .$ Since $\mathcal{A}_{%
\left[ \eta \right] }=M,$ there exists $\xi \in T_{x}M$ such that $\left(
x,\xi \right) \in \widetilde{\mathcal{A}}_{\left[ \eta \right] }$. Then $v$
is differentiable at $x$ and $\frac{\partial L}{\partial v}\left( x,\xi
\right) -\eta =d_{x}v$ (see \cite{fa2}, Theorem 2.5)$.$ Moreover, if $\gamma
=\pi \circ \varphi _{t}^{L}\left( \mathcal{L}^{-1}\left( x,p\right) \right) $
is the curve such that $\gamma |_{\left[ 0,+\infty \right) }$ is $\left( u,L-%
\left[ \eta \right] ,\alpha \left( \left[ \eta \right] \right) \right) $%
-calibrated, then%
\begin{equation*}
v\left( \gamma \left( t\right) \right) -v\left( \gamma \left( 0\right)
\right) =A_{L-\eta +\alpha \left( \left[ \eta \right] \right) }\left( \gamma
|_{\left[ 0,t\right] }\right) \text{ for all }t>0.
\end{equation*}%
Dividing by $t,$ and letting $t\rightarrow 0^{+}$, we get%
\begin{equation*}
d_{x}v\left( \dot{\gamma}\left( 0\right) \right) =L\left( x,\dot{\gamma}%
\left( 0\right) \right) -\eta \left( \dot{\gamma}\left( 0\right) \right)
+\alpha \left( \left[ \eta \right] \right) .
\end{equation*}%
By Fenchel inequality, this can happen if and only if $\frac{\partial L}{%
\partial v}\left( x,\dot{\gamma}\left( 0\right) \right) -\eta =$ $d_{x}v.$
Therefore $\dot{\gamma}\left( 0\right) =\xi $ and $\left( x,p\right) \in
\mathcal{A}_{\left[ \eta \right] }^{\ast }=\mathcal{G}_{\eta ,u}$.
\end{proof}

Other examples of absorbing graphs are graphs contained in neighborhoods
foliated by invariant Lipschitz Lagrangian graphs. We say that an open $%
\mathcal{V}$ in $T^{\ast }M$ is foliated by invariant Lipschitz Lagrangian
graphs if each $\left( x,p\right) \in \mathcal{V}$ belongs to a unique
invariant Lipschitz Lagrangian graph $\mathcal{G}\subset \mathcal{V}$.

\begin{definition}
We say that a Tonelli Hamiltonian $H$ is locally Lipschitz integrable on an
invariant Lipschitz Lagrangian graph $\mathcal{G}_{\eta ,u}$ if there exists
a neighborhood $\mathcal{V}\subset T^{\ast }M$ of $\mathcal{G}_{\eta ,u}$
foliated by invariant Lipschitz Lagrangian graphs.
\end{definition}

Now we can to prove the Theorem \ref{teo2}:\medskip

\begin{proof}
(of Theorem \ref{teo2}) Let $\mathcal{V}$ be a neighborhood of $\mathcal{G}%
_{\eta ,u}$ foliated by invariant Lipschitz Lagrangian graphs. Note that an
invariant Lipschitz Lagrangian graph $\mathcal{G}$ contained in $\mathcal{V}$
is absorbing. Since $\mathcal{G}_{\eta ,u}\subset \mathcal{V}$, by Theorem %
\ref{teo1},%
\begin{equation*}
\mathcal{G}_{\eta ,u}=\mathcal{A}_{\left[ \eta \right] }^{\ast }=\mathcal{N}%
_{\left[ \eta \right] }^{\ast }.
\end{equation*}%
We can use the upper semicontinuity of the Mañé set (see for instance \cite%
{arn1}, Proposition 13]) to deduce that there exists $U_{0}\subset
H^{1}\left( M;%
\mathbb{R}
\right) ,$ an open neighborhood of $\left[ \eta \right] ,$ such that$\
\mathcal{N}_{c}^{\ast }\subset \mathcal{V}$ for all $c\in U_{0}.$ Given%
\begin{equation*}
h\in V=\bigcup\nolimits_{c\in U_{0}}\partial \alpha \left( c\right) ,
\end{equation*}%
let us suppose that $h$ $\in \partial \alpha \left( c_{0}\right) $ for some $%
c_{0}\in U_{0}.$ So $\mathcal{N}_{c_{0}}^{\ast }\subset \mathcal{V}$ and, by
connectedness of Mañé set, $\mathcal{N}_{c_{0}}^{\ast }$ is contained in
some invariant Lipschitz Lagrangian $\mathcal{G}_{\omega ,v}\subset \mathcal{%
V}$. Since $\mathcal{G}_{\omega ,v}$ is absorbing, it follows from Lemma \ref%
{lem4} that $c_{0}=\left[ \omega \right] $. Moreover, by the Theorem \ref%
{teo1}, $\beta $ is differentiable at $h\in \partial \alpha \left(
c_{0}\right) .$
\end{proof}

\section{Dynamical properties on a neighborhood of graphs and
differerentiability of $\protect\beta $\label{sec1}}

The tiered Mañé set, denoted by $\mathcal{N}^{T}\left( L\right) $ was
introduced by M-C. Arnaud in \cite{arn1}. It is defined as the union of Mañé
sets $\widetilde{\mathcal{N}}_{c}$ for all cohomology class $c\in
H^{1}\left( M;%
\mathbb{R}
\right) $ , i.e.%
\begin{equation*}
\mathcal{N}^{T}\left( L\right) =\bigcup\limits_{c\in H^{1}\left( M\right) }%
\widetilde{\mathcal{N}}_{c}.
\end{equation*}%
Arnaud also defines the dual tiered Mañé $\mathcal{N}_{\ast }^{T}\left(
L\right) $ as $\mathcal{L}\left( \mathcal{N}^{T}\left( L\right) \right)
\subset T^{\ast }M$ and proves, in the work \textit{A particular
minimization property implies }$C^{0}$\textit{-integrability} (see \cite%
{arn2}), that the dual tiered Mañé $\mathcal{N}_{\ast }^{T}\left( L\right) $
is whole cotangent bundle $T^{\ast }M$ if and only if $T^{\ast }M$ is
foliated by invariant Lipschitz Lagrangian graph. In this section we are
interested in the differentiability of $\beta $ when there exists a
neighborhood $\mathcal{V}$ of an invariant Lipschitz Lagrangian $\mathcal{G}$
contained in $\mathcal{N}_{\ast }^{T}\left( L\right) .$ We study the
relation of this hypothesis with the existence of absorbing graphs contained
in the neighborhood $\mathcal{V}$.

\begin{lemma}
\label{lem2}Let us assume that for $\left[ \eta \right] \in H^{1}\left( M;%
\mathbb{R}
\right) $ there exists a neighborhood $\mathcal{V}$ of $\mathcal{A}_{\left[
\eta \right] }^{\ast }$ in $T^{\ast }M$ such that $\mathcal{V\subset N}%
_{\ast }^{T}\left( L\right) .$ Hence if for $c\in H^{1}\left( M;%
\mathbb{R}
\right) ,\mathcal{A}_{c}^{\ast }\cap \mathcal{V}$ is non-empty, then $%
\mathcal{A}_{\left[ \eta \right] }\subset \mathcal{A}_{c}.$
\end{lemma}

\begin{proof}
Let $\eta _{c}$ be a representative of the cohomology class $c.$ Let us
consider $\left( x_{0},p_{0}\right) \in \mathcal{A}_{c}^{\ast }\cap \mathcal{%
V}$ and $\delta >0$ such that the ball $B_{\delta }\left( x_{0},p_{0}\right)
\subset T^{\ast }M$ centered at $\left( x_{0},p_{0}\right) $ and radius $%
\delta $ is contained in $\mathcal{V}$. By (\cite{arn2}, Proposition 10), we
can find $T_{0}\left( c\right) >0$ such that for all $T>T_{0}\left( c\right)
$ and for every Tonelli minimizing curve $\gamma :\left[ 0,T\right]
\rightarrow M$ with $\gamma \left( 0\right) =\gamma \left( T\right) $ for
the Lagrangian $L-\eta _{c},$ we have $d\left( \mathcal{L}\left( x_{0},\dot{%
\gamma}\left( 0\right) \right) ,\left( x_{0},p_{0}\right) \right) <\delta ,$
where $d$ is the distance in $T^{\ast }M.$

Since $x_{0}\in $ $\mathcal{A}_{c},$ one of characterizations of Aubry sets
(see for instance \cite{fa2}, Theorem 2.1 (4)) ensures the existence of a
sequence of Tonelli minimizing curves $\gamma _{n}:\left[ 0,T_{n}\right]
\rightarrow M$ for the Lagrangian $L-c$ with $\gamma _{n}\left( 0\right)
=\gamma _{n}\left( T_{n}\right) =x_{0}$ such that $T_{n}\rightarrow \infty $
and $A_{L-c+\alpha \left( c\right) }\left( \gamma _{n}\right) \rightarrow 0.$
Therefore, if $T_{n}>T_{0}\left( c\right) ,$ we have%
\begin{equation*}
d\left( \mathcal{L}\left( x_{0},\dot{\gamma}_{n}\left( 0\right) \right)
,\left( x_{0},p_{0}\right) \right) <\delta ,\text{ hence }\mathcal{L}\left(
x_{0},\dot{\gamma}_{n}\left( 0\right) \right) \in \mathcal{V\subset N}_{\ast
}^{T}\left( L\right) .
\end{equation*}%
In particular, $\left( \gamma _{n},\dot{\gamma}_{n}\right) $ is a periodic
orbit contained in some Mañé set $\widetilde{\mathcal{N}}_{\left[ \lambda
_{n}\right] }$ for some closed 1-form $\lambda _{n}.$ Therefore $\left(
\gamma _{n},\dot{\gamma}_{n}\right) $ supports a measure $\mu _{n}$ which is
$\left[ \lambda _{n}\right] $-minimizing.

Now we will show that $\mathcal{A}_{\left[ \eta \right] }\subset \mathcal{A}%
_{c}.$ Indeed, let us consider $y\in \mathcal{A}_{\left[ \eta \right] }$ and
$q\in T_{y}^{\ast }M$ such that $\left( y,q\right) \in \mathcal{A}_{\left[
\eta \right] }^{\ast }.$ Let $\epsilon >0$ such that $B_{\epsilon }\left(
y,q\right) \subset \mathcal{V}$. Again by (\cite{arn2}, Proposition 10), we
can find $T_{0}\left( \left[ \eta \right] \right) >0$ such that for all $%
T>T_{0}\left( \left[ \eta \right] \right) $ and for every Tonelli minimizing
curve $\zeta :\left[ 0,T\right] \rightarrow M$ with $\zeta \left( 0\right)
=\zeta \left( T\right) $ for the Lagrangian $L-\left[ \eta \right] ,$ we
have $d\left( \mathcal{L}\left( y,\dot{\zeta}\left( 0\right) \right) ,\left(
y,q\right) \right) <\epsilon .$

Take $n$ sufficiently large such that $T_{n}>\max \left\{ T_{0}\left(
c\right) ,T_{0}\left( \left[ \eta \right] \right) \right\} .$ Because of
Tonelli Theorem, we know that for each $T_{n},$ there exists a Tonelli
minimizing curve $\Gamma _{n}:\left[ 0,T_{n}\right] \rightarrow M$ with $%
\Gamma _{n}\left( 0\right) =\Gamma _{n}\left( T_{n}\right) =y$ for the
Lagrangian $L-\eta $ which is homologous to $\gamma _{n}.$ Since $%
T_{n}>T_{0}\left( \eta \right) ,$ we have $d\left( \mathcal{L}\left( y,\dot{%
\Gamma}_{n}\left( 0\right) \right) ,\left( y,q\right) \right) <\epsilon $
which implies
\begin{equation*}
\mathcal{L}\left( y,\dot{\Gamma}_{n}\left( 0\right) \right) \in B_{\epsilon
}\left( y,q\right) \subset \mathcal{V}.
\end{equation*}%
As a consequence, $\left( \Gamma _{n},\dot{\Gamma}_{n}\right) $ is a
periodic orbit which supports a measure $\nu _{n}$ which is $\overline{%
\lambda }_{n}$-minimizing for some closed 1-form $\overline{\lambda }_{n}.$
However, we have
\begin{equation*}
\rho \left( \mu _{n}\right) =\frac{1}{T_{n}}\left[ \gamma _{n}\right] =\frac{%
1}{T_{n}}\left[ \Gamma _{n}\right] =\rho \left( \nu _{n}\right) .
\end{equation*}%
Thus,%
\begin{eqnarray*}
A_{L-\lambda _{n}}\left( \mu _{n}\right) &=&A_{L}\left( \mu _{n}\right)
-\left\langle \rho \left( \mu _{n}\right) ,\lambda _{n}\right\rangle \\
&=&\beta \left( \rho \left( \mu _{n}\right) \right) -\left\langle \rho
\left( \nu _{n}\right) ,\lambda _{n}\right\rangle \\
&=&A_{L-\lambda _{n}}\left( \nu _{n}\right) .
\end{eqnarray*}%
Or,%
\begin{equation*}
\int\nolimits_{0}^{T_{n}}L\left( \gamma _{n},\dot{\gamma}_{n}\right)
-\lambda _{n}\left( \dot{\gamma}_{n}\right)
dt=\int\nolimits_{0}^{T_{n}}L\left( \Gamma _{n},\dot{\Gamma}_{n}\right)
-\lambda _{n}\left( \dot{\Gamma}_{n}\right) dt.
\end{equation*}%
Therefore,
\begin{eqnarray*}
A_{L-\eta _{c}+\alpha \left( c\right) }\left( \Gamma _{n}\right)
&=&A_{L-\lambda _{n}+\left( \lambda _{n}-\eta _{c}\right) +\alpha \left(
c\right) }\left( \Gamma _{n}\right) =A_{L-\lambda _{n}}\left( \Gamma
_{n}\right) +\left[ \lambda _{n}-\eta _{c}\right] \left[ \Gamma _{n}\right]
+\alpha \left( c\right) T_{n} \\
&=&A_{L-\lambda _{n}}\left( \gamma _{n}\right) +\left[ \lambda _{n}-\eta _{c}%
\right] \left[ \gamma _{n}\right] +\alpha \left( c\right) T_{n}=A_{L-\eta
_{c}+\alpha \left( c\right) }\left( \gamma _{n}\right) \rightarrow 0.
\end{eqnarray*}%
This shows that $y\in \mathcal{A}_{c}$ and we obtain $\mathcal{A}_{\left[
\eta \right] }\subset \mathcal{A}_{c}$.
\end{proof}

The following corollary relates a neighborhood contained in the dual tiered
Mañé to the existence of absorbing graphs for each cohomology class
belonging to an open subset of $H^{1}\left( M;%
\mathbb{R}
\right) .$

\begin{corollary}
\label{coro1}Let $\mathcal{V}$ be a neighborhood of an invariant Lipschitz
Lagrangian graph $\mathcal{G}_{\eta ,u}$ such that $\mathcal{G}_{\eta ,u}=%
\mathcal{A}_{\left[ \eta \right] }^{\ast }.$ Let us assume that $\mathcal{V}$
is contained in the dual tiered Mañé $\mathcal{N}_{\ast }^{T}\left( L\right)
.$ Then there exists a neighborhood $U_{0}\subset H^{1}\left( M;%
\mathbb{R}
\right) $ of $\left[ \eta \right] $ such that for each $c\in U_{0},$ there
exists an absorbing graph $\mathcal{G}_{\eta _{c},u_{c}}$ with $\left[ \eta
_{c}\right] =c$ and $u_{c}\in C^{1,1}.$ Hence $\beta $ is differentiable at
any point of $V=\bigcup\nolimits_{c\in U_{0}}\partial \alpha \left( c\right)
.$
\end{corollary}

\begin{proof}
We can use the upper semicontinuity of the Mañé set (see for instance \cite%
{arn1}, Proposition 13) to deduce that there exists $U_{0}\subset
H^{1}\left( M;%
\mathbb{R}
\right) ,$ a neighborhood of $\left[ \eta \right] ,$ such that $\ \mathcal{N}%
_{c}^{\ast }\subset \mathcal{V}$ for all $c\in U_{0}.$ In particular, for
all $c\in U_{0}$ we have $\mathcal{A}_{c}^{\ast }\subset \mathcal{V}$ and,
by the Lemma \ref{lem2}, we conclude that $\mathcal{A}_{c}=\mathcal{A}_{%
\left[ \eta \right] }.$ Moreover, since $\mathcal{G}_{\eta ,u}=\mathcal{A}_{%
\left[ \eta \right] }^{\ast },$ we have $\mathcal{A}_{c}=M$. Therefore,
there exists a closed 1-form $\eta _{c}$ with $\left[ \eta _{c}\right] =c$
and a function $u_{c}\in C^{1,1}$ such that $\mathcal{A}_{c}^{\ast }=%
\mathcal{G}_{\eta _{c},u_{c}}.$ It remains to show that $\mathcal{G}_{\eta
_{c},u_{c}}$ is absorbing. In fact, if $\left( x,p\right) \in \mathcal{L}%
\left( \mathcal{N}_{+}^{T}\left( L\right) \right) $ and
\begin{equation*}
\omega \left( x,p\right) \subset \mathcal{G}_{\eta _{c},u_{c}}=\mathcal{A}%
_{c}^{\ast },
\end{equation*}%
then the curve $\gamma \left( t\right) =\pi \circ \varphi _{t}^{H}\left(
x,p\right) ,$ the projection of the Hamiltonian flow $\varphi _{t}^{H},$
intersects $\mathcal{V}$ for some time $\tau >0.$ In particular $\varphi
_{\tau }^{H}\left( x,p\right) $ belongs to some Mañé set $\mathcal{N}_{%
\overline{c}}^{\ast }.$ This implies that $\omega \left( x,p\right) \subset
\mathcal{A}_{\overline{c}}^{\ast },$ so $\mathcal{A}_{\overline{c}}^{\ast
}\cap \mathcal{A}_{c}^{\ast }\neq \emptyset .$ The Lemma \ref{lem2} implies
that $\mathcal{A}_{\overline{c}}=M$ and, as a consequence, there exist $v\in
C^{1,1}$ and $\eta _{\overline{c}}$ a representative of the cohomology class
$\overline{c}$ such that $\mathcal{A}_{\overline{c}}^{\ast }=\mathcal{G}%
_{\eta _{\overline{c}},v}.$

By Proposition 6 of \cite{mas1}, $c$ and $\overline{c}$ belong to the same
flat $F$ of $\alpha $. Moreover, if $c_{0}$ is in the relative interior of $%
F,$ then $\mathcal{A}_{c_{0}}^{\ast }\subset \mathcal{A}_{c}^{\ast }\cap
\mathcal{A}_{\overline{c}}^{\ast }.$ It follows from Lemma \ref{lem2} that $%
\mathcal{A}_{c_{0}}=M\ $and, as a consequence, there exist a closed 1-form $%
\eta _{c_{0}}$ with $\left[ \eta _{c_{0}}\right] =c_{0}$ and a function $%
u_{0}\in C^{1,1}$ such that $\mathcal{A}_{c_{0}}^{\ast }=\mathcal{G}_{\eta
_{c_{0}},u_{0}}.$ This shows that
\begin{equation*}
\mathcal{G}_{\eta _{c_{0}},u_{0}}=\mathcal{G}_{\eta _{c},u}=\mathcal{G}%
_{\eta _{\overline{c}},v},
\end{equation*}%
and, by the graph property, $\mathcal{N}_{\overline{c}}^{\ast }=\mathcal{G}%
_{\eta _{\overline{c}},v}=\mathcal{G}_{\eta _{c},u}$. This shows that $%
\varphi _{t}^{H}\left( x,p\right) $ belongs to $\mathcal{G}_{\eta _{c},u}$
for all $t\in
\mathbb{R}
.$ In particular, $\left( x,p\right) $ belongs to $\mathcal{G}_{\eta _{c},u}$
and we conclude that $\mathcal{G}_{\eta _{c},u}$ is an absorbing graph.

The conclusion that $\beta $ is differentiable at any point of $%
V=\bigcup\nolimits_{c\in U_{0}}\partial \alpha \left( c\right) $ follows
from Theorem \ref{teo1}$.$
\end{proof}

The union of the invariant absorbing graphs obtained in the previous
corollary may or may not be an open subset of $T^{\ast }M.$ However, this is
the case, with the additional hypothesis $\dim H^{1}\left( M;%
\mathbb{R}
\right) =\dim M$, as shown in the following theorem.

\begin{theorem}
\label{teo3}Suppose that $\dim H^{1}\left( M;%
\mathbb{R}
\right) =\dim M.$ Let $U_{0}\subset H^{1}\left( M,%
\mathbb{R}
\right) $ be a neighborhood of $\left[ \eta \right] $ such that for each $%
c\in U_{0},$ there exists an invariant Lipschitz Lagrangian absorbing graph $%
\mathcal{G}_{\eta _{c},u_{c}}$ with $\left[ \eta _{c}\right] =c$ and $%
u_{c}\in C^{1,1}.$ Then the Hamiltonian is locally Lipschitz integrable on $%
\mathcal{G}_{\eta ,u},$ where $u=u_{\left[ \eta \right] }.$
\end{theorem}

\begin{proof}
Let us define the map%
\begin{eqnarray}
&&F:M\times U_{0}\rightarrow T^{\ast }{}M  \label{f3} \\
&&\left( x,c\right) \rightarrow \left( x,\eta _{c}\left( x\right)
+d_{x}u_{c}\right)  \notag
\end{eqnarray}%
where the closed 1-form $\eta _{c}$ is a representative of $c$ and $u_{c}\in
C^{1,1}$ such that $\mathcal{G}_{\eta _{c},u_{c}}$ is an invariant absorbing
graph. In particular, by Theorem \ref{teo1}, we have $\mathcal{A}_{c}^{\ast
}=\mathcal{G}_{\eta _{c},u_{c}}.$ This map is injective because the
absorbing graphs are disjoint. Moreover, $F$ is also continuous. In fact,
let $\left( x_{n},c_{n}\right) \rightarrow \left( x_{0},c\right) $ and
consider its associated graphs $\mathcal{A}_{c_{n}}^{\ast }=\mathcal{G}%
_{\eta _{c_{n}},u_{c_{n}}}^{\ast }.$ Let us consider the sequence of
Lipschtz functions $\lambda _{n}=\eta _{c_{n}}+du_{c_{n}}.$ Note that, by
the graph property, we have $\mathcal{A}_{c_{n}}^{\ast }=\mathcal{N}%
_{c_{n}}^{\ast }.$ Moreover, by upper semicontinuity of the Mañé set (see
\cite{arn1}, Proposition 13), we conclude that for $n$ sufficiently large,
the sequence $\left( x_{n},\lambda _{n}\left( x_{n}\right) \right) \in
\mathcal{N}_{c_{n}}^{\ast }$ remains in a compact set. We can conclude that
-- up to selecting a subsequence -- $\left( x_{n},\lambda _{n}\left(
x_{n}\right) \right) $ converges to some point that belongs to $\mathcal{N}%
_{c}^{\ast }.$ Moreover, since $\mathcal{N}_{c_{n}}^{\ast }=\mathcal{A}%
_{c_{n}}^{\ast }$ and $x_{n}\rightarrow x_{0},$ by the graph property, we
conclude that%
\begin{equation*}
\left( x_{n},\lambda _{n}\left( x_{n}\right) \right) \rightarrow \left(
x_{0},\eta _{c}\left( x_{0}\right) +d_{x_{0}}u_{c}\right) \Leftrightarrow
F\left( x_{n},c_{n}\right) \rightarrow F\left( x_{0},c\right) .
\end{equation*}%
The result follows from Invariance Domain Theorem for manifolds. In fact,
\begin{equation*}
\dim \left( M\times U\right) =2n=\dim T^{\ast }M
\end{equation*}%
and $F$ is continuous and injective. Therefore $F\left( M\times U\right) $
is open and $\mathcal{G}_{\eta ,u}\subset F\left( M\times U\right) ,$ that
is, the Hamiltonian is locally Lipschitz integrable on $\mathcal{G}_{\eta
,u}.$
\end{proof}

As an immediate consequence of Corollary \ref{coro1} and Theorem \ref{teo3},
we present a local version of Arnaud's Theorem (\cite{arn2}, Theorem 1).

\begin{corollary}
\label{coro2}Suppose that $\dim H^{1}\left( M;%
\mathbb{R}
\right) =\dim M.$ Let $\mathcal{V}$ be a neighborhood of an invariant
Lipschitz Lagrangian $\mathcal{G}_{\eta ,u}$ such that $\mathcal{G}_{\eta
,u}=\mathcal{A}_{\left[ \eta \right] }^{\ast }$ and $\mathcal{V}$ is
contained in the dual tiered Mañé $\mathcal{N}_{\ast }^{T}\left( L\right) .$
Then the Hamiltonian is locally Lipschitz integrable on $\mathcal{G}_{\eta
,u}.$
\end{corollary}

\section{The Case $M=\mathbb{T}^{2}$}

We say that a homology $h\in H_{1}\left( \mathbb{T}^{2};%
\mathbb{R}
\right) $ is \textit{rational }if there exists $\lambda >0$ such that $%
\lambda h\in i_{\ast }H_{1}\left( \mathbb{T}^{2};%
\mathbb{Z}
\right) ,$ where $i_{\ast }:H_{1}\left( \mathbb{T}^{2};%
\mathbb{Z}
\right) \rightarrow H_{1}\left( \mathbb{T}^{2};%
\mathbb{R}
\right) $ is the natural map$.$ Since the manifold treated in this section
is the torus $\mathbb{T}^{2},$ we can say that a homology $h\in H_{1}\left(
\mathbb{T}^{2};%
\mathbb{R}
\right) $ is \textit{irrational }if it is not rational. For general
manifolds, there is the concept of $k$\textit{-irrationality} (see for
instance \cite{mas1}).

Note that if two cohomology classes lie in the relative interior of a flat
of $\alpha ,$ by \cite{mat1} their Mather sets coincide. Then $\widetilde{%
\mathcal{M}}\left( \partial \beta (h)\right) $ denotes the Mather sets of
all the cohomologies in the relative interior of $\partial \beta (h).$ A
homology class $h$ is said to be singular if its Legendre transform $%
\partial \beta (h)$ is a singular flat, i.e. its Mather set $\widetilde{%
\mathcal{M}}\left( \partial \beta (h)\right) $ contains fixed points$.$

One natural question in the present context is if the set $%
V=\bigcup\nolimits_{c\in U_{0}}\partial \alpha \left( c\right) $ obtained in
Theorem $\ref{teo2}$ is an open of $H_{1}\left( M;%
\mathbb{R}
\right) .$ If this the case, by convexity $\beta ,$ we obtain that it is of
class $C^{1}$ in $V.$ In the case of $M=\mathbb{T}^{2},$ a positive answer
to this question is given in the next proposition.

\begin{proposition}
\label{prop2}Suppose $M=\mathbb{T}^{2}.$ Let $c_{0}\in H^{1}\left( \mathbb{T}%
^{2};%
\mathbb{R}
\right) $ and $U_{0}$ be a neighborhood of $c_{0}$ in $H^{1}\left( \mathbb{T}%
^{2};%
\mathbb{R}
\right) $ such that $\beta $ is differentiable at any point of $%
V=\bigcup\nolimits_{c\in U_{0}}\partial \alpha \left( c\right) .$ Then $%
\alpha $ is $C^{1}$ in $U_{0}.$ In particular, $V$ is an open set of $%
H_{1}\left( \mathbb{T}^{2};%
\mathbb{R}
\right) $ and $\beta $ is of class $C^{1}$ in $V.$
\end{proposition}

\begin{proof}
Let $c\in U_{0}$ be a cohomology class and let $h$ be a rational homology
class belonging to $\partial \alpha \left( c\right) .$ Suppose by
contradiction that $\partial \alpha \left( c\right) \neq \left\{ h\right\} .$
Recall that, in case $M=\mathbb{T}^{2},$ all flats of $\beta $ flats are
radial (see \cite{car2}) . Therefore, exchanging, if necessary, $h$ and a
nonzero extremal point of $\partial \alpha \left( c\right) ,$ which exists
because $\partial \alpha \left( c\right) $ has two extremal points which
also are rational, we can assume
\begin{equation*}
\partial \alpha \left( c\right) =\left[ \lambda h,h\right] =\left\{ th:t\in %
\left[ \lambda ,1\right] \right\}
\end{equation*}%
for some $\lambda <1.$ Let us consider a sequence $t_{n}>1$ such that $%
t_{n}\rightarrow 1$. We will show that, for $n$ sufficiently large, $\beta $
is differentiable at $t_{n}h.$ In fact, there exists a sequence $c_{n}\in
H^{1}\left( \mathbb{T}^{2};%
\mathbb{R}
\right) $ such that $c_{n}\in \partial \beta \left( t_{n}h\right) .$Thus
\begin{equation*}
\beta \left( t_{n}h\right) =\left\langle t_{n}h,c_{n}\right\rangle -\alpha
\left( c_{n}\right) .
\end{equation*}%
The sequence $c_{n}$ has subsequence bounded. Otherwise, let us consider a
subsequence of $\frac{c_{n}}{\left\Vert c_{n}\right\Vert }$ convergent, i.e.
$\frac{c_{n_{k}}}{\left\Vert c_{n_{k}}\right\Vert }\rightarrow a$, where $%
\left\Vert .\right\Vert $ denotes a norm on $H^{1}\left( \mathbb{T}^{2};%
\mathbb{R}
\right) $ (see for instance \cite{fa1}, Section 4.10). So we have $%
\left\Vert c_{n_{k}}\right\Vert \rightarrow \infty $ and
\begin{equation*}
\frac{\alpha \left( c_{n_{k}}\right) }{\left\Vert c_{n_{k}}\right\Vert }%
=\left\langle t_{n_{k}}h,\frac{c_{n_{k}}}{\left\Vert c_{n_{k}}\right\Vert }%
\right\rangle -\frac{\beta \left( t_{n_{k}}h\right) }{\left\Vert
c_{n_{k}}\right\Vert }\rightarrow \left\langle h,a\right\rangle ,
\end{equation*}%
which contradicts the superlinearity of $\alpha .$ Then we can assume,
extracting a subsequence if necessary, $c_{n}\rightarrow \widetilde{c}.$
Thence%
\begin{equation*}
\beta \left( h\right) =\lim \beta \left( t_{n}h\right) =\lim \left\langle
t_{n}h,c_{n}\right\rangle -\alpha \left( c_{n}\right) =\left\langle h,%
\widetilde{c}\right\rangle -\alpha \left( \widetilde{c}\right) .
\end{equation*}%
This shows that $\widetilde{c}\in \partial \beta \left( h\right) .$ Since $%
\beta $ is differentiable at $h,$ we conclude that $\widetilde{c}=c$ and $%
c_{n}\in U_{0}$ for $n$ sufficiently large. Therefore $\beta $ is
differentiable at $t_{n}h.$

Observe that $t_{n}h$ is non-singular for all $n\in
\mathbb{N}
$. Otherwise, there exists a fixed point, which comprise the support of a
minimizing measure $\mu _{0},$ in $\widetilde{\mathcal{M}}\left( \partial
\beta \left( t_{n}h\right) \right) =\widetilde{\mathcal{M}}_{c_{n}}.$ Since $%
\rho \left( \mu _{0}\right) =0,$ the set $\left[ 0,t_{n}h\right] $ is
contained in a flat of $\beta $ and the maximal flat $\left[ \lambda h,h%
\right] $ may be extended because $t_{n}>1.$

Therefore we can apply (\cite{sor1}, Corollary 1) to deduce that $\mathbb{T}%
^{2}$ is foliated by periodic orbits and $\mathcal{M}_{c_{n}}=\mathcal{A}%
_{c_{n}}=\mathbb{T}^{2}.$ By (\cite{gon3}, Proposition 2.1) we can consider
a point $\left( x,v\right) \in T\mathbb{T}^{2}$ such that the orbit $\varphi
_{t}\left( x,v\right) $ is periodic with period $T$ and which comprise the
support of a minimizing measure with rotation vector $\lambda h.$ If $\left(
x,v\right) $ is a fixed point, take $T=+\infty .$ Since $x\in \mathcal{A}%
_{c_{n}},$ there exists a only $v_{n}$ such that $\left( x,v_{n}\right) \in
\widetilde{\mathcal{A}}_{c_{n}}$ is a periodic point with period $T_{n}.$ By
semicontinuity of the Aubry set, $\left( x,v_{n}\right) $ converges to some
point $\left( x,w\right) $ of $\widetilde{\mathcal{A}}_{c}.$ By the graph
property, $w=v.$

We now prove that for all $\delta >0,$ there exists $n_{0}\in
\mathbb{N}
$ such that $T_{n_{0}}\geq T-\delta $. In fact, if some $\delta >$ $0$ is
such that $0<T_{n}<T-\delta $ for all $n\in
\mathbb{N}
,$ extracting a subsequence if necessary, $T_{n}\rightarrow S<T$ and $%
\varphi _{T_{n}}\left( x,v_{n}\right) \rightarrow \varphi _{S}\left(
x,v\right) =\left( x,v\right) $ which contradicts the minimality of period
of the orbit of $\left( x,v\right) $. Let us consider $h_{0}\in H_{1}\left(
\mathbb{T}^{2};%
\mathbb{Z}
\right) $ such that the probability measure carried by the orbit $\varphi
_{t}\left( x,v\right) $ have homology $\frac{1}{T}h_{0}=\lambda h$ and the
probability measure carried by the orbit $\varphi _{t}\left( x,v_{n}\right) $
have homology $\frac{1}{T_{n}}h_{0}=t_{n}h.$ Given $\delta =T\left(
1-\lambda \right) ,$ there exists $n_{0}\in
\mathbb{N}
$ such that $T_{n_{0}}\geq T-T\left( 1-\lambda \right) =T\lambda .$ Then,
since $h_{0}=T\lambda h$ and $h_{0}=T_{n_{0}}t_{n_{0}}h,$ we have $\frac{%
T\lambda }{T_{n_{0}}}=t_{n_{0}}\leq 1$ which is a contradiction.

Now suppose that $h$ be an irrational homology class belongs to $\partial
\alpha \left( c\right) .$ Therefore any probability measure $h$-minimizing
is supported on a lamination of the torus without closed leaves. Moreover,
this measure is uniquely ergodic. In particular, $h$ is not contained in any
non-trivial flat of $\beta .$
\end{proof}

As a consequence of the above Proposition, we present a local version of (%
\cite{mas1}, Theorem 3).

\begin{theorem}
\label{teo4}Suppose that $M=\mathbb{T}^{2}$ and let $\left[ \eta \right] $
be a cohomology class. Then the following statements are equivalent:\newline
(1) There exists $u\in C^{1,1}\left( M\right) $ and a neighborhood $\mathcal{%
V}$ of $\mathcal{G}_{\eta ,u}$ in $T^{\ast }M$ such that $\mathcal{A}_{\left[
\eta \right] }^{\ast }=\mathcal{G}_{\eta ,u}$ and $\mathcal{V}\subset
\widetilde{\mathcal{N}}^{T}\left( L\right) .$ \newline
(2) $\beta $ is differentiable at any point of $V=\bigcup\limits_{c\in
U_{0}}\partial \alpha \left( c\right) $ for some neighborhood $U_{0}\subset
H^{1}\left( M;%
\mathbb{R}
\right) $ of $\left[ \eta \right] .$\newline
(3) There exists $u\in C^{1,1}\left( M\right) $ such that the Hamiltonian is
locally Lipschitz integrable on $\mathcal{G}_{\eta ,u}.$
\end{theorem}

\begin{proof}
$(1)\Rightarrow (2)$ It follows immediately of Corollary \ref{coro1}%
.\medskip \newline
$(2)\Rightarrow \left( 3\right) $ By Proposition \ref{prop2} $V$ is open in $%
H_{1}\left( \mathbb{T}^{2};%
\mathbb{R}
\right) $ and the Legendre transform is a homeomorphism between $V$ and $%
U_{0}.$ For each $h\in V,$ let us consider $c_{h}=\partial \beta \left(
h\right) .$ Then the set of the classes $c_{h}$ with $h$ rational and
non-singular is dense in $U_{0}.$ In fact, the set of the classes $c_{h}$
with $h$ rational (see \cite{mas1}, Lemma 7) is dense in $H^{1}\left(
\mathbb{T}^{2};%
\mathbb{R}
\right) $. Now note that zero is the only (possibly) singular class belongs
to $\partial \alpha \left( c\right) $ such that $c\in U_{0}$, because if a
non-zero class $h$ is singular belongs to $\partial \alpha \left( c\right) $
such that $c\in U_{0}$, then there is a fixed point in the Mather set of $c$%
. Thus $\partial \alpha \left( c\right) $ contains the homology of the Dirac
measure on the fixed point. This contradicts the Proposition \ref{prop2},
which says that $\alpha $ is $C^{1}$ in $U_{0}.$

It follows from (\cite{sor1}, Corollary 1) that $\mathcal{A}_{c_{h}}=\mathbb{%
T}^{2}$ is foliated by periodic orbits. By semicontinuity of Aubry set, we
have that $\mathcal{A}_{c}=\mathbb{T}^{2}$ for all $c\in U_{0}$ and that
there exist $\eta _{c}$ a representative of class $c$ and $u_{c}\in
C^{1,1}\left( \mathbb{T}^{2}\right) $ such that $\mathcal{A}_{c}^{\ast }=%
\mathcal{G}_{\eta _{c},u_{c}}.$ In particular, $\mathcal{A}_{\left[ \eta %
\right] }^{\ast }=\mathcal{G}_{\eta ,u}$ for some $u\in C^{1,1}.$ Moreover,
these graphs are absorbing. Indeed, this follows from assumption that $\beta
$ is differentiable at any point $V=\bigcup\limits_{c\in U_{0}}\partial
\alpha \left( c\right) $ and from converse of Theorem \ref{teo1}. Therefore,
since $\dim H_{1}\left( \mathbb{T}^{2};%
\mathbb{R}
\right) =\dim
\mathbb{R}
^{2},$ by the Theorem \ref{teo3}, we obtain a neighborhood of $\mathcal{G}%
_{\eta ,u}$ foliated by invariant Lipschitz Lagrangian graphs. \medskip
\newline
$\medskip (3)\Rightarrow \left( 1\right) $ All point of $\mathcal{V}$
belongs to some invariant Lipschitz Lagrangian graphs. In particular,
belongs to some Mañé set, so $\mathcal{V}\subset \widetilde{\mathcal{N}}%
^{T}\left( L\right) .$ Since graphs of a foliation are absorbing, by Theorem %
\ref{teo1} we have $\mathcal{A}_{\left[ \eta \right] }^{\ast }=\mathcal{G}%
_{\eta ,u}.$
\end{proof}

\section{\ An Example: vertical exact magnetic Lagrangian}

In this section we present a Lagrangian on the two torus $\mathbb{T}^{2}$
such that the $\beta $-function is of class $C^{1}$ in the open set
\begin{equation*}
A=\left\{ \left( h_{1},h_{2}\right) \in H_{1}\left( \mathbb{T}^{2};%
\mathbb{R}
\right) :h_{1}\neq 0\right\}
\end{equation*}%
and is not differentiable at any point outside of $A.$ Let us consider the
magnetic Lagrangian on the two torus $\mathbb{T}^{2}$ defined by
\begin{equation*}
L\left( x,y,v_{1},v_{2}\right) =\frac{\left\Vert v\right\Vert ^{2}}{2}%
+\left\langle \left( 0,\cos \left( 2\pi x\right) \right) ,v\right\rangle
\end{equation*}%
where the metric $\left\Vert .\right\Vert $ is induced by inner product.
This type of convex and superli-near Lagrangian is an example of vertical
magnetic Lagrangian, apresented in \cite{car1}, in which the authors were
interested in flats of $\beta $ function. Here we are interested in the
differentiability of $\beta $ and consequently in flats of $\alpha .$

The Euler-Lagrange flow associated with this Lagrangian is generated by the
vector field:
\begin{equation*}
X_{L}:\left\{
\begin{array}{l}
\dot{x}=v \\
\dot{v}=-2\pi \sin \left( 2\pi x\right) Jv%
\end{array}%
\right.
\end{equation*}%
where $J$ is the $2\times 2$ canonical sympletic matrix. Since the energy
function for $L$ is $E\left( x,y,v\right) =\frac{1}{2}\left\Vert
v\right\Vert ^{2},$ for each energy level $E>0,$ we can consider the angle $%
\varphi $ (with horizontal line) of trajectories of the Euler-Lagrange flow.
This means that $\varphi $ is the new parameter of $v=\left(
v_{1},v_{2}\right) :$
\begin{equation*}
v_{1}=\sqrt{2E}\cos \varphi ,v_{2}=\sqrt{2E}\sin \varphi .
\end{equation*}%
It is easy to see that $H\left( x,\varphi \right) =\cos \left( 2\pi x\right)
+\sqrt{2E}\sin \varphi $ is a first integral. The critical points of $H$ are
$\left( 0,\frac{\pi }{2}\right) $-maximum, $\left( 0,-\frac{\pi }{2}\right) $%
-saddle, $\left( \frac{1}{2},\frac{\pi }{2}\right) $-saddle and $\left(
\frac{1}{2},-\frac{\pi }{2}\right) $-minimum.

Depending of the level of energy, there exist or not invariant Lipschitz
Lagrangian graphs. A suficiently condition for existing invariant Lipschitz
Lagrangian graphs in the level of energy $E$ is $E>\frac{1}{2}.$ In fact, if
$\left\vert F\right\vert <\sqrt{2E}-1,$ the level $H^{-1}\left( F\right) $
of $H$ is composed by two graphs. This follows from Implicit Function
Theorem. Indeed, $\frac{\partial H}{\partial \varphi }=\sqrt{2E}\cos \varphi
=0$ if and only if $\varphi =\pm \frac{\pi }{2}.$ In this case, $\cos \left(
2\pi x\right) \pm \sqrt{2E}=F$ implies $F\geq \sqrt{2E}-1$ or $F\leq 1-\sqrt{%
2E},$ which contradicts $\left\vert F\right\vert <\sqrt{2E}-1.$ This also
show that the two graphs of $\varphi _{1}$ and $\varphi _{2}$ with $\varphi
_{2}=\pi -\varphi _{1},$ given implicitly by equation%
\begin{equation*}
\cos \left( 2\pi x\right) +\sqrt{2E}\sin \varphi =F
\end{equation*}%
are invariant for $E>\frac{1}{2}$ and $\left\vert F\right\vert <\sqrt{2E}-1.$

\begin{figure}[tbp]
\centering
\includegraphics{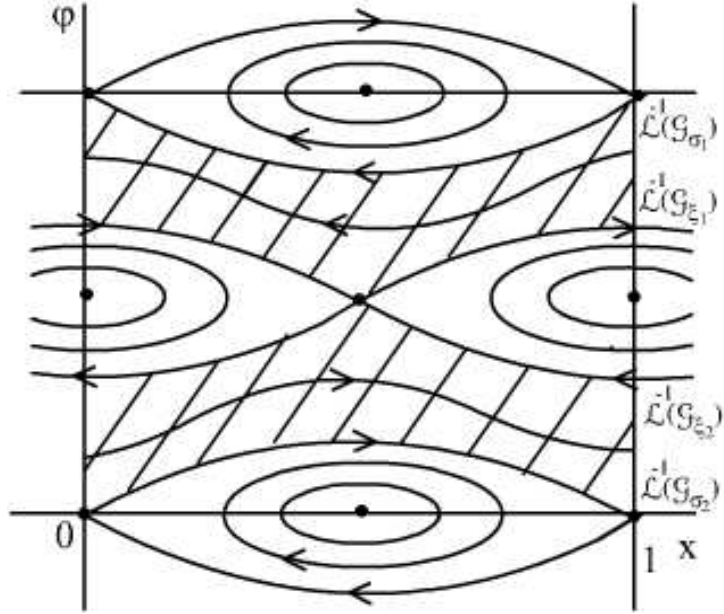} \label{fig}
\caption{Energy level $E>\frac{1}{2}$}
\end{figure}

The figure 1 describes the projection these graphs in the section $x\varphi ,
$ in the energy level $E>\frac{1}{2}.$ Let us consider the closed 1-forms%
\begin{equation}
\eta _{i}=\sqrt{2E}\cos \varphi _{i}\left( x\right) dx+Fdy,i=1,2,
\label{fo3}
\end{equation}%
dependent of $E$ and $F.$ Since $\mathcal{L}\left( x,v\right) =\left(
x,\left\langle \left( v_{1},v_{2}+\cos \left( 2\pi x\right) \right) ,\cdot
\right\rangle \right) ,$ we have $\mathcal{G}_{\eta _{i},0}=\mathcal{L}%
\left( x,\sqrt{2E}\cos \varphi _{i},\sqrt{2E}\sin \varphi _{i}\right) $.
These graphs compose a local foliation in $T^{\ast }\mathbb{T}^{2}$, then
are absorbing graphs. It follows from Theorem \ref{teo1} that $\beta $ is
differentiable at any point of $\partial \alpha \left( \left[ \eta _{i}%
\right] \right) ,$ where
\begin{equation}
\left[ \eta _{i}\right] =\int_{0}^{1}\sqrt{2E}\cos \varphi _{i}\left(
x\right) dx+Fdy,  \label{fo4}
\end{equation}%
for all $i=1,2,$ $E>\frac{1}{2}$ and $F$ with $\left\vert F\right\vert <%
\sqrt{2E}-1.$ Moreover, by Proposition \ref{prop2}, we conclude that $\alpha
$ function is differentiable at $\left[ \eta _{i}\right] .$

Take $F\rightarrow \sqrt{2E}-1$ by left and $F\rightarrow 1-\sqrt{2E}$ by
right in (\ref{fo4}), for each $i=1,2$, we obtain four closed 1-forms $%
\sigma _{1},\sigma _{2},\xi _{1}$ and $\xi _{2}$ whose graphs are the
connections of saddle. Therefore $\mathcal{G}_{\sigma _{1},0}\cap \mathcal{G}%
_{\sigma _{1},0}\neq \emptyset $ and the graphs $\mathcal{G}_{\sigma _{1},0}$
and $\mathcal{G}_{\sigma _{2},0}$ is not absorbing. Moreover, since $%
\mathcal{G}_{\eta _{i},0}$ are absorbing graphs, $\mathcal{A}_{\left[ \eta
_{i}\right] }^{\ast }=\mathcal{G}_{\eta _{i},0}.$ Then, by semicontinuity of
Aubry set, we obtain $\mathcal{A}_{\left[ \sigma _{i}\right] }=\mathbb{T}^{2}
$, so $\mathcal{A}_{\left[ \sigma _{i}\right] }^{\ast }=\mathcal{G}_{\sigma
_{i},0}.$ It follows from (\cite{mas1}, Proposition 6) that $\left[ \sigma
_{1}\right] $ and $\left[ \sigma _{2}\right] $ belong to same flat of $%
\alpha .$ We obtain, analogously, that $\left[ \xi _{1}\right] $ and $\left[
\xi _{2}\right] $ belong to same flat of $\alpha .$ This show that $\beta $
is neither differentiable at points of $\partial \alpha \left( \left[ \sigma
_{i}\right] \right) $ nor at points of $\partial \alpha \left( \left[ \xi
_{i}\right] \right) .$ Actually, the homology classes which belong to $%
\partial \alpha \left( \left[ \sigma _{i}\right] \right) $ and $\partial
\alpha \left( \left[ \xi _{i}\right] \right) $ are $\pm \sqrt{2E}\left(
0,1\right) $. In fact, the intersection $\mathcal{A}_{\left[ \sigma _{1}%
\right] }\cap \mathcal{A}_{\left[ \sigma _{2}\right] }$ is the intersection
of two graphs given by implicit equation\
\begin{equation*}
\cos \left( 2\pi x\right) +\sqrt{2E}\sin \varphi =\sqrt{2E}-1.
\end{equation*}%
Thence $\mathcal{A}_{\left[ \sigma _{1}\right] }\cap \mathcal{A}_{\left[
\sigma _{2}\right] }$ is the closed curve $\gamma _{1}:t\mapsto \left( \frac{%
1}{2},\sqrt{2E}t\right) .$ In particular, $\gamma _{1}$ belongs to $\mathcal{%
M}_{\left[ \sigma _{1}\right] }\cap \mathcal{M}_{\left[ \sigma _{2}\right] }$
which supports the measure $\mu _{1}$ with rotation vector equals $\rho
\left( \mu _{1}\right) =\sqrt{2E}\left( 0,1\right) .$ Analogously, $\mathcal{%
A}_{\left[ \xi _{1}\right] }\cap \mathcal{A}_{\left[ \xi _{2}\right] }$ is
the closed curve $\gamma _{2}:t\mapsto \left( 0,-\sqrt{2E}t\right) $ which
supports a probability measure $\mu _{2}$ with $\rho \left( \mu _{2}\right)
=-\sqrt{2E}\left( 0,1\right) .$ Therefore we prove that $\beta $ function is
not differentiable at any homology class of the form $\left( 0,h_{2}\right) .
$

\begin{remark}
Since $H$ and $E$ are first integrals, given two constants $a$ and $b,$ the
level set $\left( H,E\right) ^{-1}\left( a,b\right) =\left\{ \left(
x,v\right) :H\left( x,v\right) =a,E\left( x,v\right) =b\right\} $ is an
absorbing set. Therefore, it follows from Proposition \ref{pro3} that if $%
H\left( \widetilde{\mathcal{A}}_{c}\right) =F$ for some $c\in H^{1}\left(
\mathbb{T}^{2};%
\mathbb{R}
\right) ,$ then $\left( H,E\right) ^{-1}\left( F,\alpha \left( c\right)
\right) $ projects onto the whole torus $\mathbb{T}^{2}.$
\end{remark}

By previous remark, for all $c\in H^{1}\left( \mathbb{T}^{2};%
\mathbb{R}
\right) $ the Aubry set $\widetilde{\mathcal{A}}_{c}$ is contained in an
invariant Lipschitz Lagrangian graph. Since the supports of minimizing
measures contained in the graphs $\mathcal{G}_{\sigma _{i},0}$ and $\mathcal{%
G}_{\xi _{i},0}$ have vector rotation of the form $\left( 0,h_{2}\right) ,$
if $h\in A$ then $\widetilde{\mathcal{M}}^{h}$ is contained in a
neighborhood foliated by Lipschitz Lagrangian graphs. It follows from
Theorem \ref{teo4} and Proposition \ref{prop2} that $\beta $ is of class $%
C^{1}$ in some neighborhood of $h$ and hence of class $C^{1}$ in $A.$


\addcontentsline{toc}{chapter}{Bibliografia}

\begin{thebibliography}{99}
\bibitem{arn1} Arnaud, M-C., \textit{The tiered Aubry set for autonomous
Lagrangian functions,} Ann. Inst. Fourier (Grenoble) 58, no. 5, 1733--1759
(2008).

\bibitem{arn2} Arnaud, M-C., \textit{A particular minimization property
implies }$C^{0}$\textit{-integrability, }Journal of Differential Equations,
250 (5), 2389-2401 (2011).

\bibitem{ber1} Bernard, P., Contreras, G., \textit{A generic property of
families of Lagrangian systems, }Annals of Mathematics, Pages 1099-1108 from
Volume 167 (2008).

\bibitem{ber2} Bernard, P., \textit{On the Conley Decomposition of Mather
sets}, Rev. Mat. Iberoamericana, vol. 26, no. 1, pp. 115--132 (2010).

\bibitem{car2} Carneiro, M. J., \textit{On minimizing measures of the action
of autonomous Lagrangians,} Nonlinearity, 8 (6): 1077-1085, (1995).

\bibitem{car1} Carneiro, M. J., Lopes, A., \textit{On the minimal action
function of autonomous lagrangians associated to magnetic fields}, Annales
de l'I. H. P., section C, tome 16, N.6, 667-690, (1999).

\bibitem{gon2} Contreras, G., Iturriaga, R., \textit{Global Minimizers of
Autonomous Lagrangians}, CIMAT, México, (2000).

\bibitem{gon3} Contreras, G., Macarini, L., Paternain, G., \textit{Periodic
Orbits for Exact Magnetic Flows on Surfaces, }International Mathematics
Research Notices, No. 8, (2004).

\bibitem{fa1} Fathi, A., \textit{Weak KAM Theorem and Lagrangian Dynamics
Preliminary Version Number 10}, (2008).

\bibitem{fa2} Fathi, A., Figalli A., Rifford L., \textit{On the Hausdorff
dimension of the Mather quotient.} Comm. Pure Appl. Math. 62, no. 4,
445-500, (2009).

\bibitem{fa3} Fathi, A., Giuliani, A., Sorrentino, A., \textit{Uniqueness of
Invariant Lagrangian Graphs in a Homology or a Cohomology Class. }Ann.
Scuola Norm. Sup. Pisa Cl. Sci. (5) Vol. VIII, 659-680, (2009).

\bibitem{man2} Mañé, R., \textit{Generic properties and problems of
minimizing measure of Lagrangian dynamical systems, }Nonlinearity, 9, N.2,
273-310, (1996).

\bibitem{man1} Mañé, R., \textit{Global Variational Methods in Conservative
Dynamics}, IMPA, (1993).

\bibitem{mas1} Massart, D., \textit{On Aubry sets and Mather's action
functional,} Israel J. Math., 134, 157--71 (2003).

\bibitem{mas2} Massart, D., \textit{Vertices of Mather's Beta function, II, }%
Ergodic Theory Dynam. Systems, 29, no. 4, 1289--1307, (2009).

\bibitem{sor1} Massart, D., Sorrentino, A.\textit{, Differentiability of
Mather's average action and integrability on closed surfaces, }Nonlinearity,
24, 1777--1793, (2011).

\bibitem{mat1} Mather, J. N., \textit{Action minimizing invariant measures
for positive definite Lagrangian Systems}, Math. Zeitschrift, 207, 169-207,
(1991).
\end{thebibliography}
\end{document}